\documentclass[12pt]{amsart}
\usepackage{amscd,amsmath,amsthm,amssymb}
\usepackage{amsfonts,amssymb,amscd,amsmath,enumerate,verbatim}
\usepackage[left]{lineno}

\def\NZQ{\mathbb}               

\def\ZZ{{\NZQ Z}}
\def\RR{{\NZQ R}}

\def\G{{\mathcal G}}
\def\F{{\mathcal F}}

\def\Cc{{\mathcal C}}

\def\Fc{{\mathcal F}}
\def\ab{{\mathbf a}}

\def\xb{{\mathbf x}}

\def\cb{{\mathbf c}}

\def\vb{{\mathbf v}}
\def\opn#1#2{\def#1{\operatorname{#2}}} 
\opn\chara{char} \opn\length{\ell} \opn\pd{pd} \opn\rk{rk}
\opn\projdim{proj\,dim} \opn\injdim{inj\,dim} \opn\rank{rank}
\opn\depth{depth} \opn\grade{grade} \opn\height{height}
\opn\embdim{emb\,dim} \opn\codim{codim}
\opn\Cl{Cl}

\opn\Tr{Tr} \opn\bigrank{big\,rank}
\opn\superheight{superheight}\opn\lcm{lcm}
\opn\trdeg{tr\,deg}
	\opn\reg{reg} \opn\lreg{lreg} \opn\ini{in} \opn\lpd{lpd}
	\opn\size{size} \opn\sdepth{sdepth}
	\opn\link{link}\opn\fdepth{fdepth}\opn\lex{lex}
	\opn\tr{tr}
	\opn\type{type}
	\opn\gap{gap}
	\opn\arithdeg{arith-deg}
	\opn\revlex{revlex}
	\opn\div{div} \opn\Div{Div} \opn\cl{cl} \opn\Cl{Cl}
	\opn\Spec{Spec} \opn\Supp{Supp} \opn\supp{supp} \opn\Sing{Sing}
	\opn\Ass{Ass} \opn\Min{Min}\opn\Mon{Mon}
	\opn\Ann{Ann} \opn\Rad{Rad} \opn\Soc{Soc}
	\opn\Im{Im} \opn\Ker{Ker} \opn\Coker{Coker} \opn\Am{Am}
	\opn\Hom{Hom} \opn\Tor{Tor} \opn\Ext{Ext} \opn\End{End}
	\opn\Aut{Aut} \opn\id{id}
	\def\F{{\mathcal F}}
	\opn\nat{nat}
	\opn\pff{pf}
	\opn\Pf{Pf} \opn\GL{GL} \opn\SL{SL} \opn\mod{mod} \opn\ord{ord}
	\opn\Gin{Gin} \opn\Hilb{Hilb}\opn\sort{sort}
	\opn\PF{PF}\opn\Ap{Ap}
	\opn\mult{mult}
	\opn\bight{bight}
	\opn\div{div}
	\opn\Div{Div}
	\opn\aff{aff}
	\opn\relint{relint} \opn\st{st}
	\opn\lk{lk} \opn\cn{cn} \opn\core{core} \opn\vol{vol}  \opn\inp{inp} \opn\nilpot{nilpot}
	\opn\link{link} \opn\star{star}\opn\lex{lex}\opn\set{set}
	\opn\width{wd}
	\opn\Fr{F}
	\opn\QF{QF}
	\opn\G{G}
	\opn\type{type}\opn\res{res}
	\opn\conv{conv}
	\opn\Deg{Deg}
	\opn\Sym{Sym}
	\opn\Con{Con}
	\opn\gr{gr}
	
	\def\pot#1#2{#1[\kern-0.28ex[#2]\kern-0.28ex]}

	\opn\dirlim{\underrightarrow{\lim}}
	\opn\inivlim{\underleftarrow{\lim}}
	\let\union=\cup
	\let\sect=\cap

	\let\iso=\cong
	
	\let\Sect=\bigcap
	\let\Dirsum=\bigoplus
	
	\let\to=\rightarrow
	
	\def\Implies{\ifmmode\Longrightarrow \else
		\unskip${}\Longrightarrow{}$\ignorespaces\fi}
	\def\implies{\ifmmode\Rightarrow \else
		\unskip${}\Rightarrow{}$\ignorespaces\fi}
	\def\iff{\ifmmode\Longleftrightarrow \else
		\unskip${}\Longleftrightarrow{}$\ignorespaces\fi}

	\let\:=\colon
	\newtheorem{Theorem}{Theorem}[section]
	\newtheorem{Lemma}[Theorem]{Lemma}
	\newtheorem{Corollary}[Theorem]{Corollary}
	\newtheorem{Proposition}[Theorem]{Proposition}

	\newtheorem{Example}[Theorem]{Example}

	\let\epsilon\varepsilon
	\let\kappa=\varkappa
	\def\qed{\ifhmode\textqed\fi
		\ifmmode\ifinner\quad\qedsymbol\else\dispqed\fi\fi}
	\def\textqed{\unskip\nobreak\penalty50
		\hskip2em\hbox{}\nobreak\hfil\qedsymbol
		\parfillskip=0pt \finalhyphendemerits=0}
	\def\dispqed{\rlap{\qquad\qedsymbol}}
	
	\opn\dis{dis}
	\def\pnt{{\raise0.5mm\hbox{\large\bf.}}}
	
	\opn\Lex{Lex}
	
\begin{document}
		
		\title{Toric rings attached to simplicial complexes}
		
		\author {J\"urgen Herzog, Somayeh Moradi and Ayesha Asloob Qureshi}

		\address{J\"urgen Herzog, Fachbereich Mathematik, Universit\"at Duisburg-Essen, Campus Essen, 45117
			Essen, Germany} \email{juergen.herzog@uni-essen.de}

		\address{Somayeh Moradi, Department of Mathematics, Faculty of Science, Ilam University,
			P.O.Box 69315-516, Ilam, Iran}
		\email{so.moradi@ilam.ac.ir}
		
			\address{Ayesha Asloob Qureshi, Sabanci University, Faculty of Engineering and Natural Sciences, Orta Mahalle, Tuzla 34956, Istanbul, Turkey}
	
		\email{aqureshi@sabanciuniv.edu}
		
		\dedicatory{ }
		\keywords{toric rings, quasi-forests, normal rings, canonical module}
		\subjclass[2010]{Primary 13A02; 13P10, Secondary 05E40}
		\thanks{The first and the second author gratefully acknowledge the hospitality provided by the Faculty of Engineering and Natural Sciences, Sabanci University  during their visit. Somayeh Moradi is supported by the Alexander von Humboldt Foundation. Ayesha Asloob Qureshi is supported by The Scientific and Technological Research Council of Turkey - T\"UBITAK (Grant No: 122F128).}
\begin{abstract}
We consider standard graded toric rings $R_{\Delta}$ whose generators correspond to the faces of a simplicial complex $\Delta$.  
When $R_{\Delta}$ is normal, it is shown that its divisor class group is free. For a flag complex $\Delta$ which is the clique complex of a perfect graph, a nice description for the class group and the canonical module of $R_{\Delta}$ in terms of the minimal vertex covers of the graph is given. Moreover, for a quasi-forest simplicial complex a quadratic Gr\"obner basis for the defining ideal of $R_{\Delta}$  is presented. Using this fact we give  combinatorial descriptions for the $a$-invariant and the Gorenstein property of $R_\Delta$.
\end{abstract}

		\maketitle
		
		\setcounter{tocdepth}{1}
\section*{Introduction}		

It is common to associate with a simplicial complex $\Delta$ its  Stanley-Reisner ring,  whose defining relations are the squarefree monomials corresponding to the non-faces of $\Delta$. It turned out that algebraic invariants of the Stanley-Reisner ring provide deep insight into the combinatorial structure of the underlying simplicial complex, as exemplified by Stanley's proof of the upper bound conjecture (see for example \cite{BH}). In this paper we propose to associate with $\Delta$ a standard graded toric ring $R_\Delta$,  and to study its algebraic properties in terms of the combinatorics of $\Delta$. 

For a long time  similar  approaches have already been considered  in numereous papers,  when toric rings have been attached to combinatorial objects like graphs, posets, matroids and polymatroids. Many references to such results can be  found  for example in the monographs \cite{HH}, \cite{HHO},  \cite{MSt} and \cite{V}. The key questions in all these cases is to identify the binomial relations of the corresponding toric rings, to classify those combinatorial objects whose toric ring is normal and consequently is Cohen-Macaulay, and to find out when the associated  toric ring is even Gorenstein,  which usually  meets  nice symmetry properties of the underlying combinatorial object.

The toric ring  $R_\Delta$ which we study here is defined as follows. The set of  all facets of a simplicial complex $\Delta$ will be  denoted by $\mathcal{F}(\Delta)$. If $\Fc(\Delta)=\{F_1,\ldots,F_m\}$, then we write $\Delta=\langle F_1,\ldots,F_m\rangle$. Let  $X=\{x_1,\ldots,x_n\}$ be the vertex set of $\Delta$,  and let  $K$ be a field. For a non-empty set $F\subseteq X$, we define the monomial  $x_F=\prod_{x_i\in F} x_i$  in the polynomial ring $K[x_1,\ldots,x_n]$ and for $F=\emptyset$, we set $x_F=1$. Then  $R_\Delta$ is defined to be  the subalgebra $K[x_Ft: F\in \Delta]$ 
of the polynomial ring $S=K[x_1,\ldots,x_n,t]$. The algebra $R_\Delta$	
has a $K$-basis consisting  of monomials of $S$. If $v=x_1^{a_1}\cdots x_n^{a_n}t^k$ belongs to $R_\Delta$, we set $\deg v=k$. This grading makes $R_\Delta$ a standard graded $K$-algebra. Our  rings $R_\Delta$ belong to the class of toric rings of lattice polytopes,  which  include, for example, Hibi rings arising from posets, edge rings and toric rings of stable set polytopes arising from finite graphs.

It is an important but difficult problem to characterize those simplicial complexes $\Delta$ for which $R_\Delta$ is normal.   Indeed, we show  in 
Proposition~\ref{oddcycle}  that if $\Delta$ is a $1$-dimensional simplicial complex and   $R_\Delta$ is normal, then $G_\Delta$ must satisfy the odd cycle condition, see \cite{HO}.  However assuming $R_\Delta$ is normal, we  are mainly interested  in  computing   its  divisor class group $\Cl(R_\Delta)$ and to  determine the class of the canonical module in   $\Cl(R_\Delta)$, 




A general criterion for the normality of the toric ring of a lattice polytope is given for example in \cite[Theorem~4.5]{HHO}, and a description of the divisor  class group of an affine  normal semigroup in terms of generators and relations is given by Chouinard~\cite{Ch} (see also\cite[Corollary 4.56]{BG}),  and in Villarreal's book \cite[Theorem 9.8.19]{V}. Matsushita ~\cite {M} uses these results to give sufficient conditions for the class group of a normal toric ring of a lattice polytope to be torsionfree.

When $\Delta$ is a flag complex, i.e., $\Delta$ is the independence complex of a graph $G$, the ring $R_{\Delta}$ is the toric ring $K[\mathcal{Q}_G]$ of the stable set polytope $\mathcal{Q}_G$ of $G$. It is known that if $G$ is a perfect graph, then the defining ideal of $K[\mathcal{Q}_G]$ has a squarefree initial ideal with respect to any reverse lexicographic order, and in particular $K[\mathcal{Q}_G]$ is normal, see~\cite{HO1}. 
Moreover, it was shown in \cite[Proposition~3.1]{HM} that for a perfect graph $G$ the divisor class groups of the ring $K[\mathcal{Q}_G]$ is torsionfree. 

There are other cases known in which $R_\Delta$ is normal. This is for example the case when $\Delta$ is a one dimensional simplicial complex whose facets form a cycle,  as shown in \cite[Theorem~8.1]{EN},  or when $\Delta$ is a matroid. 


We now describe the results of this paper.  One of the main results in Section~\ref{sec:1} is Theorem~\ref{thm:classgroup}, where it is shown that if $R_\Delta$ is normal, then  $\Cl(R_{\Delta})$ is generated by the classes of the minimal prime ideals $P$ of $(t)$ and the unique generating relation among these generators is described precisely. The proof of this theorem is based on Nagata's theorem~\cite[Theorem~4.52]{BG}.
In view of Theorem~\ref{thm:classgroup} it is of interest to explicitly identify the minimal prime ideals $P$ of $(t)$ and the coefficients of each class $[P]$ in the generating relation of $\Cl(R_\Delta)$. Let $C$ be a minimal vertex cover of $G_\Delta$, and let $P_C=(x_Ft\: F\in \Delta_C)$, where $\Delta_C$ is the restriction of $\Delta$ to the set $C$. In Theorem~\ref{hope} together with Proposition~\ref{min} it is shown that the ideals $P_C$ are minimal prime ideals of $(t)$. In general, however, the prime ideals $P_C$ are not the only minimal prime ideals of $(t)$, as it is shown by an example. It follows from \cite[Corollary 4.34] {BG} that the minimal prime ideals of $(t)$ are all monomial prime ideals, and they are determined by the  facets of the cone spanned by the lattice points corresponding to the generators of $R_\Delta$.  In Corollary~\ref{rankclass} we then show that $\Cl(R_{\Delta})$ is free of rank $r-1$, where $r$ is the number of minimal prime ideals of $(t)$.
 

It is of interest to determine all simplicial complexes $\Delta$ for which the ideals of the form $P_C$, where $C$ is a minimal vertex cover of $G_{\Delta}$, are all the minimal prime ideals of $(t)$. This characterization  is given in  Theorem~\ref{healthybread}, where it is shown that the prime ideals $P_C$  with $C$ a minimal vertex cover of $G_\Delta$  are precisely the minimal prime ideals of $R_\Delta$  if and only if $\Delta$ is a flag complex and $G_{\Delta}$ is a perfect graph. When this is the case,
 Theorem~\ref{healthybread} enables us to recover a result in \cite{HM}, which shows that $\Cl(R_{\Delta})$ is free of rank $r-1$, where $r$ is the number of minimal vertex covers of $G_{\Delta}$
 (see Corollary~\ref{perfectclass}).
Moreover, using Theorem~\ref{healthybread} we can describe the class of the canonical module $\omega_{R_{\Delta}}$ of $R_{\Delta}$.

If $(t)$ is a radical ideal in a normal ring $R_{\Delta}$, then the $a$-invariant of $R_\Delta$ has a nice combinatorial interpretation in terms of the minimum number of faces of $\Delta$ which cover the vertex set of $\Delta$ (see Theorem~\ref{ainvariant}). Such situation happens for example when $\Delta$ is a quasi-forest.

In Section~\ref{sec:2} we focus our attention to quasi-forests.
It can be seen from Dirac's characterization of chordal graphs that $\Delta$ is a quasi-forest if and only if it is the clique complex of the chordal graph $G_\Delta$ (see \cite[Theorem 3.3]{HHZ1}).   
Hence the toric ring of a quasi-forest $\Delta$ is precisely the toric ring of the stable set polytope of the cochordal graph $(G_{\Delta})^c$. Since any cochordal graph is a perfect graph, it is known that the defining ideal of $R_{\Delta}$ has a squarefree initial ideal with respect to any reverse lexicographic order (see  \cite{HO1}). 
As the main result of Section~\ref{sec:2} in  Theorem~\ref{powers} a quadratic Gr\"obner basis for the defining ideal of $R_\Delta$ is presented. This implies that $R_\Delta$ is Koszul.
Using the given Gr\"obner basis we show in Proposition~\ref{radical} that $(t)$ is a radical ideal.
 We  give an example,  which shows that for general simplicial complexes, the ideal  $(t)$ need not to be a radical ideal.

Assuming that $\Delta$ is a quasi-forest, one can apply Theorem~\ref{verygood} and Theorem~\ref{ainvariant}  to determine the class of the canonical ideal and  the $a$-invariant of $R_\Delta$. 
Since $R_\Delta$ is Gorentein if and only if 
$[\omega_{R_{\Delta}}]=0$, we obtain from Theorem~\ref{verygood} and   \cite{HHZ} the following equivalent conditions for the Gorensteinness of $R_\Delta$:
\begin{enumerate}
	\item[(i)] $R_\Delta$ is Gorenstein. 
		\item[{(ii)}]  $G_\Delta$ is unmixed. 
	\item[(iii)] $G_\Delta$ is Cohen-Macaulay.
	\item[(iv)]	$[n]$ is the disjoint union of facets of $\Delta$ which admit a free vertex.
\end{enumerate}

During the preparation of the paper we used Normaliz~\cite{BIRS} and Macaulay2~\cite{GS} to consider specific examples, and we appreciate very much having these tools available. 

\section{On the  class group of toric  rings of simplicial complexes}\label{sec:1}

In this section we study the class group of normal toric rings of simplicial complexes. 
Moreover, we investigate the  height one monomial prime ideals in an arbitrary toric ring $R_{\Delta}$ which are required to understand the class group and the class of the canonical module of $R_{\Delta}$. 
Throughout this section, $\Delta$ is a simplicial complex on the vertex set $[n]$.

\begin{Theorem}
	\label{thm:classgroup}
	Let $R=R_\Delta$ be a normal domain, and let $P_1,\ldots,P_r$ be the minimal prime ideals of $(t)$, where $t\in R$ is the element corresponding to the empty face of $\Delta$. Then $\Cl(R)$ is generated by the classes $[P_i]$, $i=1,\ldots,r$. Since $R_{P_i}$ is a discrete valuation ring, we have $tR_{P_i}=P_i^{a_i}R_{P_i}$ with $a_i\in \ZZ$ for $i=1,\ldots,r$. Then  $\sum_{i=1}^ra_i[P_i]=0$ is the only generating relation among these generators of $\Cl(R_\Delta)$. 
\end{Theorem}

\begin{proof}
	We set $R=R_\Delta$. It is obvious that $R_t=K[t,t^{-1}, x_1,\ldots,x_n]$. In particular, $R_t$ is a factorial domain. Hence Nagata's theorem \cite[Theorem~4.52]{BG} implies that $\Cl(R)$ is generated by the elements $[P_i]$ , $i=1,\ldots,r$.  We denote by $\div(I)$ the divisor of a divisional ideal. Since $\div(t)=\sum_{i=1} ^ra_i\div(P_i)$, we see that $\sum_{i=1}^ra_i[P_i]=0$. 
	
	Next we show  that this is the only generating relation among the generators $[P_i]$ of $\Cl(R)$. Indeed, suppose $\sum_{i=1}^rb_i[P_i]=0$ with $b_i\in \ZZ$. Then there exists $g\in Q(R)$, the quotient field of $R$, such that $\sum_{i=1}^rb_i\div(P_i)=\div(g)$.
	Since each  $\div(P_i)$ is mapped to $0$ under the canonical map $\Div(R)\to \Div(R_t)$, it follows  that $\div(g)$ is also mapped to $0$, which means that $gR_t=R_t$. In other words, $g$ is a unit in $R_t=K[t,t^{-1}, x_1,\ldots,  x_n]$.  Thus, $g=\lambda t^a$ for some $\lambda\in K\setminus\{0\}$ and $a\in \ZZ$, since all units  in $R_t$ are of this form.  It follows that $\div(g)=\div(t^a)=a\div(t)=a\sum_{i=1}^ra_i\div(P_i)$. Since the elements $\div(P_i)$ are linearly independent in $\Div(R)$, we see that $b_i=aa_i$ for all $i$, as desired. 
\end{proof}

Our next goal is to identify height one monomial prime ideals in $R_{\Delta}$. First we investigate those which contain $t$, or in other words the minimal prime ideals of $(t)$ in $R_\Delta$. Since $(t)$ is a monomial ideal, its minimal prime ideals are monomial prime ideals, see \cite[Corollary 4.34]{BG}. 

\begin{Lemma}
	\label{firstmin}
	Let $P\subset R_{\Delta}$ be a prime ideal containing  $t$, and let $C=\{i\: x_it\in P\}$. Then $C$ is a vertex cover of $G_\Delta$.
\end{Lemma}

\begin{proof}
	Suppose $C$ is not a vertex cover of $G_\Delta$. Then there exists an edge $\{i,j\}$ of $G_\Delta$ with  $\{i,j\}\sect C=\emptyset$, and we have $(x_it)(x_jt)=(t)(x_ix_jt)\in P$. This is a contradiction, since $x_it, x_jt\not\in P$.
\end{proof}




Before proving the next result, we recall a few facts about affine semigroups and semigroup algebras. 
We identify the monomials $\xb^{\ab}t^b\in K[x_1,\ldots,x_n,t]$  with their exponent vectors $(\ab,b)\in \ZZ^{n+1}$. Then the monomial $K$-basis of $R_\Delta$ corresponds to an affine semigroup $S\subset \ZZ^{n+1}$ which is generated by the  lattice points  $p_F=\sum_{i\in F}e_i+e_{n+1}$ in $\ZZ^{n+1}$. Here $e_1,\ldots,e_{n+1}$ is the  standard basis of $\ZZ^{n+1}$. 

We use the notation introduced  in \cite{BH} and denote by  $\ZZ S$ the smallest subgroup of $\ZZ^{n+1}$  containing $S$ and by $\RR_+ S\subset \RR^{n+1}$ the smallest cone containing $S$.  In our case, $\ZZ S=\ZZ^{n+1}$. Since $R$ is normal,  Gordon's lemma \cite[Proposition 6.1.2]{BH} guaranties that $S=\ZZ^{n+1}\sect \RR_+ S$. 

A  hyperplane $H$, defined as the set of solutions  of   the linear equation $f(\xb):=a_1x_1+\cdots +a_{n+1}x_{n+1}=0$, is a supporting hyperplane of the cone $\RR_+ S$,  if $H\sect  \RR_+ S\neq \emptyset$ and $f(\xb)\geq 0$ for all $\cb\in \RR_+ S$. A subset $\F$  of  $\RR_+ S$ is called a face of  $\RR_+ S$, if there exists  a supporting hyperplane $H$ of $\RR_+ S$  such that $\F=H\sect  \RR_+ S$. 

By \cite[Corollary 4.35]{BG} all minimal prime ideals of a monomial ideal in $R_\Delta$ are monomial prime ideals. In particular, they are generated by subsets of the generators $x_Ft$ of $R_\Delta$. Moreover,  it follows from  \cite[Proposition 2.36 and Proposition 4.33]{BG} that  $P\subset R_\Delta$ is a monomial prime ideal of $R_\Delta$  if and only if there exists a face  $\F$ of $\RR_+ S$ such that $P =(x_Ft\: p_F\not\in \F)$. In other words, $P$ is a monomial prime ideal, if and only if there exists a supporting hyperplane $H$ of $ \RR_+ S$ such  that
\[
P=(x_Ft\: F\in \Delta \text{ and } f(p_F)>0),
\]
where $f$ is a   linear form defining $H$.

	The supporting hyperplane $H$ of a facet is uniquely determined. Since $H$ is spanned by lattice points, a linear form $f=\sum_{i=1}^{n+1}c_ix_i$ defining $H$ has rational coefficients. By clearing denominators we may assume that all $c_i$ are integers, and then dividing $f$ by the greatest common divisor of the $c_i$, we may furthermore assume that $\gcd(c_1,\ldots,c_{n+1})=1$. Then this normalized linear form $f$ is uniquely determined by $H$. It has the property that $f(H)=0$ and $f(\ZZ^{n+1})\sect \ZZ_{\geq 0}=\ZZ_{\geq 0}$. Indeed, since $\gcd(c_1,\ldots,c_{n+1})=1$, there exist $p= (b_1,\ldots,b_{n+1})\in \ZZ^{n+1}$ with $\sum_{i=1}^{n+1}b_ic_i=1$, which implies that $f(p)=1$.

	If $P$ is  a height $1$ monomial prime ideal, then  $P=(x_Ft\: p_F\not\in \F)$,  where $\F$ is a facet of $\RR_+S$. Let $H$ be the supporting hyperplane of $\F$. Then we  call the normalized linear form which defines $H$, the {\em support form} associated  to $P$.

\medskip
For a subset $W\subseteq [n]$, we set $\Delta_W=\{F\in \Delta: \ F\subseteq W\}$.  

\begin{Theorem}
	\label{hope} 
	Let $C$ be a vertex cover of $G_\Delta$. Then $P_C=(x_Ft\: F\in \Delta_C)$ is a prime ideal in $R_{\Delta}$.
\end{Theorem}

\begin{proof}
	Let $C$ be a vertex cover of $G_\Delta$. If $C=[n]$,  then 
	$P_C$ is the graded maximal ideal of $R_\Delta$, and hence a prime ideal. 
	We may therefore assume that $C\neq [n]$, and consider the hyperplane $H$ defined  by the equation $f(x)=0$, where $f(x) =-\sum_{i\not\in C}x_i+x_{n+1}$. 
	Now let $F\in \Delta$. Suppose first that  $F\in \Delta_C$. Then 
	\[
	f(p_F)=f(\sum_{i\in F}e_i+e_{n+1})=\sum_{i\in F}f(e_i)+f(e_{n+1})=1,
	\]
	since $f(e_i)=0$ for all $i\in C$ and $f(e_{n+1})=1.$
	
	Next suppose that $F\not \in \Delta_C$. Then arguing as before, we see that
	\[
	f(p_F)=-|F\setminus C|+1.
	\] 
	
	We show that $|F\setminus C|=1$. Indeed, since $F\nsubseteq C$, we have  $|F\setminus C|\geq 1$. Suppose $|F\setminus C|> 1$.  Then there exist $i,j\in F\setminus  C$ with  $i\neq j$. Since $i,j\in F$, it follows that $\{i, j\}$ is an edge of $G_{\Delta}$, contradicting the fact that $C$ is a vertex cover of $G_{\Delta}$ and $i,j\not\in C$.

	Now since $|F\setminus C|=1$, we have
	$f(p_F)\geq 0$ for all $F\in \Delta$, and that $f(p_F)=0$ if and only if $F\not\in \Delta_C$. This shows that $(x_Ft\: F\in \Delta_C)$ is a prime ideal. 
\end{proof}


\begin{Proposition}\label{min}
	The ideal $P_C$ is a minimal prime ideal of $(t)$ in $R_{\Delta}$ if and only if $C$ is a minimal vertex cover of $G_\Delta$. 
\end{Proposition}

\begin{proof}
	Suppose that $C$ is not a minimal vertex cover of $G_\Delta$. Then there exists a minimal vertex cover $C'$ of $G_\Delta$ properly contained in $ C$. It follows from Theorem~\ref{hope} that $P_{C'}$ is a prime ideal with $t \in P_{C'}$. Moreover,  $P_{C'} \subsetneq P_C$, a contradiction. 
	
	Now assume that $C$ is a minimal vertex cover of $G_\Delta$. In order to prove that  $P_C$ is a minimal prime ideal of $(t)$, we need to show that $\{p_F : F \in \Delta_C\} = S\setminus \mathcal{F}$ where $\mathcal{F}$ is a facet of the cone $\RR^+S$.  Since $\RR^+S$ is of dimension $n+1$, the facet we are looking for should have dimension $n$.  By the proof of Theorem~\ref{hope}, a supporting hyperplane $H$ for $P_C$ is given by $f(x)=0$ where $f(x)=-\sum_{i \notin C}x_i +x_{n+1}$.  Let $\mathcal{F} = \RR^+S \cap H$. We have to show that $\mathcal{F}$ has dimension $n$. In other words, we need to find $n$ points in $\mathcal{F}$ whose position vectors are linearly independent. 
	
	For $i \in [n]\setminus C$, let $p_i$ be the point whose position vector is $e_i+e_{n+1}$. Since $C$ is a minimal vertex cover, for each vertex $i \in C$, there exists a vertex $j_i \in [n] \setminus C$ such that $\{i,j_i\}$ is an edge of $G_{\Delta}$. For each edge $\{i,j_i\}$, let  $q_i$ be the point whose position vector is $e_i+e_{j_i}+e_{n+1}$. These points $p_i$ for $i \in [n]\setminus C$ and the points $q_i$ with $i \in C$ all belong to $\mathcal{F}$. It is obvious that their corresponding position vectors are linearly independent.
\end{proof}

The following example shows that in general not all the minimal monomial prime ideals of $(t)$ are  of the form $P_C$, where  $C$ is a minimal vertex cover of $G_\Delta$. 

\begin{Example}\label{notalways}{\em 
		Let $\Delta$ be the simplicial complex with vertex set $[3]$ and $F(\Delta)=\{\{1,2\}, \{2,3\}, \{1,3\}\}$. Then $P=(t, x_1t,x_2t,x_3t)$ is a minimal prime ideal of $(t)$ in $R_{\Delta}$ but $P$ is not of the form $P_C$. }
\end{Example}

	The following lemma, which   is useful in the context of Theorem~\ref{thm:classgroup}, results 
	from \cite[Remark 1.72]{BG} and the discussions  on page 148 in \cite{BG} .   
	
	\begin{Lemma}
		\label{vr}
		Let $R=R_\Delta$ be normal,  and  let $P$ be a monomial  prime ideal of $R$ of height one. Furthermore,  let $f$ be the support form  associated with $P$, and let $\vb_f=(c_1,\ldots,c_{n+1})$ be the coefficient vector of $f$. Then  the following holds:
		
		If $u\in Q(R)$ is a monomial with exponent vector $\vb_u$, then $uR_P=P^aR_P$, where 
		$
		a=\langle \vb_f,\vb_u\rangle.
		$
		Here $\langle -,-\rangle$ denotes the standard inner product in  $\RR^{n+1}$. 
	\end{Lemma}

	Now we have nice interpretation of the exponents $a_i$ appearing in Theorem~\ref{thm:classgroup}.
	
	\begin{Corollary}\label{interpretation}
		Let $R=R_\Delta$ be normal,  and let $P$ be a minimal prime ideal  of $(t)$. Furthermore,  let $f=\sum_{i=1}^{n+1}c_ix_i$ be the support form associated with $P$. Then $tR_P=P^{c_{n+1}}R_P$. 
	\end{Corollary}
	
	In combination with Theorem~\ref{thm:classgroup} we obtain
	
	\begin{Corollary}
		\label{rankclass}
		Let $R=R_\Delta$ be normal. Then $\Cl(R)$ is free of rank $r-1$, where $r$ is the number of minimal prime ideals of $(t)$ (which is also the number of facets of $\RR_+S$ which do not contain the lattice point $e_{n+1})$.
	\end{Corollary}
	
	\begin{proof}
		Let $P_1,\ldots,P_r$ be the minimal prime ideals of $(t)$ with associated support forms $f_i$. By Theorem~\ref{thm:classgroup}  and Corollary~\ref{interpretation} we have 
		\[
		\Cl(R)\iso \Dirsum_{i=1}^r\ZZ[P_i]/g\ZZ,
		\]
		where $g=\sum_{i=1}^rf_i(e_{n+1})[P_i]$. 
		
		Among the minimal prime ideals of $(t)$ are the prime ideals $P_C$ whose associated  support form is $f_C(x) =-\sum_{i\not\in C}x_i+x_{n+1}$, see Theorem~\ref{hope} and Proposition~\ref{min}. It follows that the coefficient of $[P_i]$ in $g$ is one if $P_i=P_C$ for some $i$. This implies that the coefficient vector of $g$ is unimodular, and therefore the inclusion map $g\ZZ\to \Dirsum_{i=1}^r\ZZ[P_i]$ is split injective. Hence,  $\Cl(R)\iso \ZZ^{r-1}$, as desired.  
	\end{proof}

One also needs to know the height one monomial prime ideals of $R_\Delta$ which do not contain $t$. They are given in 

\begin{Proposition}\label{othermin}
	Let $\Delta$ be a simplicial complex on $[n]$. For $i=1, \ldots, n$, let $Q_i=(x_Ft \:F \in \Delta,  i \in F)$.  Then $\{Q_1, \ldots, Q_n\}$ is the set of height one monomial prime ideal of $R_\Delta$ which do not contain $t$.
\end{Proposition}

\begin{proof}
	For $i=1, \ldots, n$, let $f_i(x)=x_i$ and $H_i=\{x\: f_i(x)=0\}$ be the hyperplane. We claim that $H_i$ is a supporting hyperplane of a facet of $\RR_+S$.  For all $F \in \Delta$, we have 
	\begin{equation*}
		f_i(p_F)=
		\begin{cases}
			1 ,& \text{if } i\in F, \\
			0 ,& \text{if } i\notin F .
		\end{cases}
	\end{equation*}
	In particular, $f_i(p_F) \geq 0$, which guarantees that $H_i$ is a supporting hyperplane, and the points $e_1, \ldots, \widehat{e_i}, \ldots, e_{n+1}$ lie on the hyperplane. This shows that $H_i$ is a supporting hyperplane of a facet of $\RR_+S$. Hence $P=(x_Ft \: f_i(p_F) >0)$ is a height one monomial prime ideal of $R_\Delta$ and $P=Q_i$.
	
	Conversely, let $P$ be a height one monomial prime ideal of $R_\Delta$ with $t \notin P$. First we show that $x_it \in P$, for some $i$. Assume that this is not the case. Then there exists some $x_Ft \in P$ with $x_Gt \notin P$, for all $G \subsetneq F$. For any $i \in F$, we have $(x_{F\setminus \{i\}}t)(x_it)=(x_Ft)(t) \in P$, but  $x_{F\setminus \{i\}}t \notin P$ and $x_it\notin P$, a contradiction. Pick any $x_it \in P$ and any $F \in \Delta$ with $ i \in F$. Then $(x_Ft)(t)=(x_it)(x_{F\setminus\{i\}}t) \in P$. Since $t \notin P$, we must have $x_Ft \in P$. This shows that $Q_i \subseteq P$. Since $P$ and $Q_i$ both have height one, we obtain $P=Q_i$. 
\end{proof}

The next step is to characterize simplicial complexes for which the height one monomial prime ideals of $R_\Delta$ are exactly those which we introduced before. 
To this end we recall the following concepts which are required for the statements of the next result 

A simplicial complex $\Delta$ is called {\em flag}, 
if all its minimal nonfaces are of cardinality two. In other words, $\Delta$ is flag if and only if $\Delta$ is the clique complex of $G_\Delta$. A finite simple graph $G$ is called {\em  perfect}, if $G$ and its complement contain no  induced odd cycle of length $>3$. Note that $G$ is perfect if and only if its complementary graph is perfect, as well. 

\begin{Theorem}
\label{healthybread}
Let $\Delta$ be a simplicial complex on $[n]$, and let $\Cc$ be the set of minimal vertex covers of $G_\Delta$. Then the following conditions are equivalent:
\begin{enumerate}
\item[{\em (i)}] $R_{\Delta}$ is normal and the set 
\[
\{P_C\: C\in \Cc\}\union \{Q_1,\ldots, Q_n\}
\]
is the set of height one monomial prime ideals of $R_\Delta$.

\item[{\em (ii)}] $\Delta$ is a flag complex and $G_\Delta$ is a perfect graph.
\end{enumerate}

\end{Theorem}

\begin{proof}
(i)\implies (ii) Suppose that $\Delta$ is not flag, and let $\Delta'$ be the clique complex of $G_\Delta$. Then $\Delta$ is a proper subcomplex of $\Delta'$. Let $S$ be the lattice points corresponding to the generators of $R_\Delta$ and $S'$ be lattice points corresponding to  the generators of $R_{\Delta'}$. Then  the cone $\RR_+S$ is stricly contained in the cone $\RR_+S'$.  

Let $\mathcal{F}_1,\ldots,\mathcal{F}_m$ be the facets of $\RR_+S$ corresponding to  the prime ideals $P_C$ and $Q_i$. It follows from   Proposition~\ref{min} and Proposition~\ref{othermin} 
that  $\mathcal{F}_1,\ldots,\mathcal{F}_m$ are also facets of $\RR_+S'$.   Then we have 
\[
\RR_+S\subsetneq \RR_+S'\subseteq \Sect_{i=1}^m \mathcal{F}_i^{(+)},
\]
where $\mathcal{F}_i^{(+)}$ is the closed half space defined by $\mathcal{F}_i$ which contains $\RR_+S'$. This shows that $\RR_+S$ must have other facets besides of $\mathcal{F}_1,\ldots,\mathcal{F}_m$.  Therefore,  $\Delta$ must be flag, and hence $\Delta$ is the clique complex of $G_\Delta$. 

It follows from \cite[Theorem 3.1]{Chv} that $G_\Delta$ must be perfect.

(ii)\implies (i)  also follows from \cite[Theorem 3.1]{Chv}. 
\end{proof}

\begin{Example} {\em 
Let $\Delta_1$ be the simplicial complex on the vertex set $[10]$ with facets
\[
F_1=\{1,2,3\}, F_2= \{4,5,6\}, F_3=\{7,8,9\}, F_4=\{3,10\}, F_5=\{6,10\}, F_6=\{9,10\},
\]
and let $\Delta_2=\Delta_1\setminus \{F_1\}$,  $\Delta_3=\Delta_2\setminus \{F_2\}$ and 
 $\Delta_4=\Delta_3\setminus \{F_3\}$. Then the graphs $G_{\Delta_i}$ are all equal and perfect, but only $R_{\Delta_1}$ and $R_{\Delta_2}$ are normal.  That $R_{\Delta_1}$ is normal follows from the  fact $\Delta_1$ is flag. That  $R_{\Delta_2}$ is normal,  we only know by calculation with Macaulay2. On the other hand, the fact that $R_{\Delta_3}$ and $R_{\Delta_4}$ are not normal follows from \cite[Proposition 1.2]{OHH} and  Proposition~\ref{oddcycle}, since both rings  contain the  combinatorial pure  subring $R_\Delta$ for which $G_\Delta$ satisfies the odd cycle condition, where  $\Delta$ is the restriction of $\Delta_3$ and $\Delta_4$ to the vertex set  $\{1,2,3,4,5,6,10\}$. 
 
The  set of minimal prime ideals of $(t)$  given  in Theorem~\ref{healthybread}(i) is  common to all $R_{\Delta_i}$, and it is only for  $i=1$ the  precise set of minimal prime ideals of  $(t)$ in $R_{\Delta_i}$, because only for  $i=1$ the simplicial complex $\Delta_i$ is flag.

}
\end{Example}

 By Theorem~\ref{healthybread} and Corollary~\ref{rankclass} we obtain the following corollary. Noting the fact that a graph $G$ is perfect if and only if its complement is perfect, one can see that the following corollary recovers a result in \cite{HM}.

\begin{Corollary}(\cite[Proposition 3.1]{HM})\label{perfectclass} Let $G$ be a perfect graph, and let $\Delta=\Delta(G)$. Then $\Cl(R_{\Delta})$ is free of rank $r-1$, where $r$ is the number of minimal vertex covers of $G$. 
\end{Corollary}

Let $\Delta$ be a flag complex on $[n]$ such that $G_{\Delta}$ is a perfect graph. Then  $R_{\Delta}$ is a Cohen-Macaulay normal domain. 
Let $\omega_{R_{\Delta}}$ be the canonical module of $R_{\Delta}$. 
By using Theorem~\ref{healthybread}  we are able to obtain a combinatorial description for the class of $\omega_{R_{\Delta}}$. 
By \cite[Corollary~3.3.19]{BH}, $\omega_{R_{\Delta}}$ is a divisorial ideal and corresponds to the relative interior of the cone $\RR_+S$, see \cite[Theorem~6.3.5(b)]{BH}.
Let $C_1, \ldots, C_r$ be the minimal vertex covers of $G_\Delta$, and let $P_i=P_{C_i}$ for $i=1, \ldots, r$. Then by Theorem~\ref{healthybread} the height one monomial prime ideals of $R_{\Delta}$ are $P_1, \ldots, P_r, Q_1, \ldots, Q_n$.
Hence by  \cite[Theorem~6.3.5(b)]{BH} we have  
\begin{equation}\label{eq:intersection}
	\omega_{R_{\Delta}}=( \bigcap_{i=1}^rP_i) \cap (\bigcap_{j=1}^n Q_j).
\end{equation} 

\begin{Theorem}\label{verygood}
Let	$\Delta$ be a flag complex on $[n]$ such that $G_\Delta$ is a perfect graph. Then with the notation introduced above, we have 
\[
	[\omega_{R_{\Delta}}]= \sum_{j=1}^r (n-|C_j|+1) [P_j].
	\]
\end{Theorem}

\begin{proof}
	To simplify the notation, we set $R=R_\Delta$. It follows from (\ref{eq:intersection}) that $[\omega_{R}]= \sum_{j=1}^r  [P_j]+\sum_{k=1}^n  [Q_k]$. 
We claim that for any $1\leq i\leq n$, 
	\[
	[x_i] = (-\sum_{i \notin C_j} [P_j])+[Q_i].
	\]
	This implies that $[Q_i]=\sum_{i \notin C_j} [P_j]$, which yields the desired conclusion. It only remains to prove the claim.  Let $[x_i]= \sum_{j=1}^r a_j [P_j]+\sum_{k=1}^n b_k [Q_k]$. 
For any prime ideals $P_j$ we know that the associated support form is $f_j(x) =-\sum_{i\not\in C_j}x_i+x_{n+1}$, see Theorem~\ref{hope} and Proposition~\ref{min}.
Then by Lemma~\ref{vr} we have
\[
a_j =\langle \vb_{f_j},e_i\rangle= \left\{
\begin{array}{ll}
	-1,  &   \text{if $i\notin C_j$,}\\
	0, & \text{if $i\in C_j$.}\\
\end{array}
\right. \]	
For any $1\leq j\leq n$, the associated support form of $Q_j$ is $f'_j(x)=x_j$ (see Proposition~\ref{othermin}). Hence 
\[
b_j =\langle \vb_{f'_j},e_i\rangle= \left\{
\begin{array}{ll}
	1,  &   \text{if $j=i$,}\\
	0, & \text{otherwise,}\\
\end{array}
\right. \]		
which completes the proof of the claim.
\end{proof}

For a given simplicial complex $\Delta$ on $[n]$, a {\em face cover} of $\Delta$ is a collection of faces $F_1, \ldots, F_r$ of $\Delta$  such that $\bigcup_{i=1}^r F_i=[n]$. The minimum size of a face cover of $\Delta$ is called the {\em face cover number} of $\Delta$, denoted by $c_{\Delta}$. 

\begin{Theorem}\label{ainvariant}
	Let $R_\Delta$ be a normal domain, and let $(t)$ be a radical ideal in $R_{\Delta}$. Then the $a$-invariant of $R_\Delta$ is $-(c_\Delta+1)$. 
\end{Theorem}

\begin{proof}
	Let  $F_1, \ldots, F_{c_{\Delta}}$ be a face cover of $\Delta$ of minimum size. We show that $u=t\prod_{i=1}^{c_{\Delta}}(x_{F_i}t) \in \omega_{R_\Delta}$. 
	Since $(t)$ is a radical ideal, it follows that $\omega_{R_\Delta}=(t)\cap (\bigcap_{j=1}^n Q_j)$. Following the definition of a face cover and $Q_j$'s, one can see that for each $j\in [n]$, there exists some $1\leq i \leq s$ such that $x_{F_i}t \in Q_j$. Hence $u \in Q_j$, for all $j$ and then $u \in \omega_{R_\Delta}$.  Furthermore, for any monomial $v\in \omega_{R_\Delta}$, we may write $v=t\prod_{i=1}^k(x_{G_i}t)$. Since $v \in Q_j$ for all $j$, we have $\bigcup_{i=1}^k G_i=[n]$. This shows that $G_1, \ldots, G_k$ is a face cover of $\Delta$, and hence $k \geq c_{\Delta}$. Hence the degree of $v$ is at least $c_{\Delta}+1$. 
\end{proof}

Toric rings of simplicial complexes are not necessarily normal. Recall that a graph $G$ is said to satisfy the {\em odd cycle condition}, if for any two induced odd cycles $C_1$ and $C_2$, either $C_1$ and $C_2$ have a common vertex or there exists $i\in V(C_1)$ and $j\in V(C_2)$ such that $\{i,j\}\in E(G)$.  In~\cite[Corollary 2.3]{HO} it is shown that the edge ring of a graph $G$ is normal if and only if $G$ satisfies the odd cycle condition. In the following proposition we show that for a  $1$-dimensional simplicial complex $\Delta$, if $G_{\Delta}$ does not satisfy the odd cycle condition, then $R_{\Delta}$ is not normal.

\begin{Proposition}\label{oddcycle}
	Let  $\Delta$ be a 1-dimensional simplicial complex. If $R_{\Delta}$ is normal, then $G_{\Delta}$ satisfies the odd cycle condition.
\end{Proposition}

\begin{proof}
	Suppose that $G_{\Delta}$ does not satisfy the odd cycle condition. Then there exist two induced odd cycles $C_1$ and $C_2$ in $G_\Delta$ with no common vertex and with no edge in $G_{\Delta}$ with endpoints in each cycle. Let $V(C_1)=\{1,2,\ldots, 2m-1\}$, $V(C_2)=\{2m,\ldots, 2n\}$, $E(C_1)=\{\{1,2\},\ldots, \{2m-2,2m-1\},\{1,2m-1\}\}$ and $E(C_2)=\{\{2m,2m+1\},\ldots, \{2n-1,2n\},\{2m,2n\}\}$. Let $u=x_1\cdots x_{2n} t^{n}$. 
	For each $i \in \Delta$, we have $x_i/1=x_it/t$ and hence $x_i$ belongs to the quotient field of $R_\Delta$. This shows that  $u$ belongs to the quotient field of $R_\Delta$. Moreover, we may write 
	\[
	u^2=[(x_1x_2t)\cdots (x_{2m-2}x_{2m-1}t) (x_{1}x_{2m-1}t)][(x_{2m}x_{2m+1}t)\cdots (x_{2n-1}x_{2n}t) (x_{2n}x_{2m}t)]. 
	\]
Hence $u^2\in R_\Delta$, while $u \notin R_{\Delta}$. Indeed, if $u \in R_{\Delta}$ then there exists  $F_1, \ldots, F_n \in \Delta$ with $u=(x_{F_1}t)\cdots (x_{F_n}t)$. This gives $F_i\cap F_j = \emptyset$ for all $i \neq j$ and $|F_i|=2$ for all $i$. The existence of such faces is possible only if $\{i,j\} \in G_\Delta$ for some $i\in V(C_1)$ and $j \in V(C_2)$, a contradiction. 
\end{proof}

\section{The  class group of toric rings attached to quasi-forests}\label{sec:2}

In this section we study toric ring of a quasi-forest. Any quasi-forest appears as the clique complex of a chordal graph, see \cite[Theorem 3.3]{HHZ1}. Hence $R_{\Delta}$ can be viewed as the toric ring of the stable set polytope of a cochordal graph. Since any cochordal graph is a perfect graph,  it is known that the defining ideal $J_{\Delta}$ of $R_{\Delta}$ has a squarefree quadratic initial ideal (see \cite{HO1}). In this section we show that $J_{\Delta}$ has indeed a quadratic Gr\"obner basis. This in particular enables us to show that $(t)$ is a radical ideal in $R_{\Delta}$.

We recall the definition of a quasi-forest. Let $\Delta$ be a simplicial complex with the vertex set $X$. A facet $F\in \Delta$ is called a {\em leaf} of $\Delta$, if there exists a facet $G\neq F$ such that for any $H\in \mathcal{F}(\Delta)$ with $H\neq F$, $F\cap H\subseteq G$. Such a facet $G$ is called a {\em branch} of $F$. A {\em free vertex} of $\Delta$ is a vertex $x\in X$ which belongs to exactly one facet of $\Delta$. It is easy to see that if $F$ is a leaf and $G$ is a branch of $F$, then any element in $F\setminus G$ is a free vertex. Let  $\Delta$ be a simplicial complex with $\mathcal{F}(\Delta)=\{F_1,\ldots,F_m\}$. Then $\Delta$ is called a {\em quasi-forest}, if there exists an order $F_1<\cdots<F_m$ such that $F_i$ is a leaf of the simplicial complex  $\langle F_1,\ldots,F_i\rangle$ for all $1\leq i\leq m$. Such an order is called a {\em leaf order} of $\Delta$.  It can be seen from Dirac's characterization of chordal graphs that $\Delta$ is a quasi-forest if and only if it is the clique complex of the chordal graph $G_\Delta$ (see \cite[Theorem 3.3]{HHZ1}).

Let $\Delta$ be a quasi-forest, and let $F_1<\cdots<F_m$ be a leaf order on $\mathcal{F}(\Delta)$. For any $1\leq i\leq m$, set $A_i= \langle F_i\rangle \setminus \langle F_1,\ldots,F_{i-1}\rangle$. Note that $F_i\in A_i$ and $\Delta=A_1\dot{\cup}\cdots\dot{\cup} A_m$.

Let $p$ be the cardinality of $\Delta$. We may write
$\Delta=\{G_1,\ldots,G_p\}$, where $G_1>\cdots>G_p$ is any total order which extends the following partial order on $\Delta$: 

For $i,j\in [m]$, if $G_r\in A_i$ and $G_s\in A_j$ and either $i<j$ or $i=j$ and $|G_r|<|G_s|$, we set $G_r>G_s$ (i.e. $r<s$).  

Consider the polynomial ring $T=K[y_1,\ldots,y_p]$ and define the surjective ring homomorphism $\pi : T \to  R_{\Delta}$ by setting
$\pi(y_i)=x_{G_i}t$ for any $1\leq i\leq p$. Let $J_{\Delta}$ be the defining ideal of $R_{\Delta}$ in $T$.

Consider the reverse lexicographic order $<$ on $T$ induced by the ordering
\[
y_1 > y_2 > \cdots > y_p.
\]

\begin{Theorem}\label{powers} 
	Let $\Delta=\{G_1,\ldots,G_p\}$ be a quasi-forest as above. Then the two types of binomials
	
	\begin{itemize}
		\item[(i)] $f_{i,j}=y_iy_j-y_ry_s$, where $G_i,G_j\in \langle F_{k}\rangle$ for some $k\in[p] $,  $G_r=G_i\cap G_j$, $G_s=G_i\cup G_j$  and $G_i$ and $G_j$ are incomparable (with respect to $\subseteq$), 
		\\
		
		\item[(ii)] $g_{i,j,x}=y_iy_j-y_ry_s$, where
		$G_i\in A_{\ell}$ and $G_j\in A_{\ell'}$ with $\ell<\ell'$, $x\in (G_i\cap F_{\ell'})\setminus G_j$,
		$G_r=G_i\setminus\{x\}$ and $G_s=G_j\cup \{x\}$ 
	\end{itemize}
	
	form a Gr\"obner basis of $J_{\Delta}$ with respect to $<$.
	In particular, $R_{\Delta}$ is Koszul and a normal Cohen-Macaulay domain. 
	
\end{Theorem}

\begin{proof}
	
	
	
	

	
		
	
	
	\medskip
	Note that for the binomials defined above, we have $\ini_<(f_{i,j})=y_iy_j$ and $\ini_<(g_{i,j,x})=y_iy_j$. Let $I_{\Delta}$ be the monomial ideal generated by all such $y_iy_j$.
	Then to prove the theorem it is enough to show that if $u$ and $v$ are two monomials is $T$ with $u,v\notin I_{\Delta}$ and $\pi(u)=\pi(v)$, then $u=v$ (see~\cite[Theorem 3.11]{HHO}).
	
	The simplicial complex $\Delta'=\langle F_1,\ldots,F_{m-1} \rangle$ is a quasi-forest as well. Note that $\Delta'=\{G_1,\ldots,G_s\}$ for some $s<p$ and $\Delta=\Delta'\dot{\cup}\{G_{s+1},\ldots,G_p\}$, where $\{G_{s+1},\ldots,G_p\}=A_m$, and each of $G_{s+1},\ldots,G_p$ contains at least one free vertex of $F_m$. Indeed, since $F_m$ is a leaf of $\Delta$, there exists $1\leq t\leq m-1$ such that $F_j\cap F_m\subseteq F_t$ for all $j\neq m$. If some $G_{i}\in \Delta_m$ does not contain a free vertex of $F_m$, then for any $x_k\in G_i$ we have $x_k\in F_{j_k}\cap F_m$ for some $j_k\neq m$. Hence $x_k\in F_t$. Therefore $G_i\subseteq F_t$, which contradicts to  $G_{i}\in \Delta_m$.
	
	Set $T'=K[y_1,\ldots,y_s]$. Considering the epimorphism  $\pi' : T' \to R_{\Delta'}$ defined by
	$\pi(y_i)=x_{G_i}t$ for any $1\leq i\leq s$, by induction on the number of facets of $\Delta$ we may assume that the set $$\{f_{i,j}:\ f_{i,j}\in T'\}\cup \{g_{i,j,x}:\  g_{i,j,x}\in T'\}$$ is a Gr\"obner basis of the defining ideal $J_{\Delta'}$ of $R_{\Delta'}$. 
	
	Let  $u=y_{i_1}\cdots y_{i_r}$  and $v=y_{j_1}\cdots y_{j_r}$
	with $1\leq i_1\leq \cdots\leq i_r\leq d$ and $1\leq j_1\leq \cdots\leq j_r\leq d$ be two monomials in $T$  with $u,v\notin I_{\Delta}$ and $\pi(u)=\pi(v)$.
	Then 
	\begin{equation}\label{eq1}
		x_{G_{i_1}}\cdots x_{G_{i_r}}=x_{G_{j_1}}\cdots x_{G_{j_r}}.  	 
	\end{equation}

	First assume that $u\in T'$. This means that $G_{i_{\ell}}\in \Delta'$ for all $1\leq \ell\leq r$. If $G_{j_{\ell}}\notin \Delta'$ for some $\ell$, then $G_{j_{\ell}}\in A_m$ and hence it contains a free vertex $x_i\in F_m$.
	Thus $x_i$ divides $x_{G_{j_1}}\cdots x_{G_{j_r}}$, while $x_i$ does not divide $x_{G_{i_1}}\cdots x_{G_{i_r}}$, a contradiction. Therefore  $G_{j_{\ell}}\in \Delta'$ for all $1\leq \ell\leq r$, as well.  Thus $v\in T'$. Moreover, since $I_{\Delta'}\subset I_{\Delta}$, we have $u,v\notin I_{\Delta'}$ and $\pi'(u)=\pi'(v)$. By the induction hypothesis, we obtain $u=v$.
	
	\medskip
	Now, assume that $u\notin T'$. The argument above shows that $v\notin T'$  as well. 
	Hence we may write $u=u'w_1$ and $v=v'w_2$ such that $u',v'$ are monomials in $T'$ and  $w_1,w_2$ are monomials in $K[y_{s+1},\ldots,y_p]$  with $w_1,w_2\neq 1$. Let $w_1=y_{i_q}\cdots y_{i_r}$ with $i_q\leq\cdots\leq i_r$ and  $w_2=y_{j_{q'}}\cdots y_{j_r}$ with $j_{q'}\leq \cdots\leq j_r$.
	Since $G_{i_q},G_{i_{q+1}},\ldots,G_{i_r}\in \langle F_m\rangle$ and $w_1\notin I_{\Delta}$, we should have
	$G_{i_q}\subseteq G_{i_{q+1}}\subseteq\cdots\subseteq G_{i_r}$. Similarly, 
	$G_{j_{q'}}\subseteq G_{j_{q'+1}}\subseteq\cdots\subseteq G_{j_r}$.
	Since $G_{i_q}\in A_m$, it contains a free vertex $x_i\in F_m$. Hence the degree of $x_i$ in $x_{G_{i_1}}\cdots x_{G_{i_r}}$ is $r-q+1$. Therefore, the degree of $x_i$ in $x_{G_{j_1}}\cdots x_{G_{j_r}}$ should be $r-q+1$ as well. Since $x_i\in F_m$ is a free vertex, $x_i\notin G_{j_{\ell}}$ for $1\leq\ell\leq q'-1$. This implies that the degree of $x_i$ in $x_{G_{j_1}}\cdots x_{G_{j_r}}$ is at most $r-q'+1$. Hence $r-q+1\leq r-q'+1$, and $q'\leq q$. By a similar argument $q\leq q'$ and then $q=q'$.
	
	We claim that $G_{i_q}=G_{j_{q}}$. By contradiction assume that there exists $x_k\in G_{i_q}\setminus G_{j_q}$. 
	Then $x_k\in G_{j_{\ell}}$ for some $\ell<q$. Otherwise the degree of $x_k$ in $x_{G_{j_1}}\cdots x_{G_{j_r}}$ is at most $r-q$, while the degree of $x_k$ in $x_{G_{i_q}}\cdots x_{G_{i_r}}$ is at least $r-q+1$, which is not possible.  
	Since $x_k\in F_m$ and $G_{j_q}\subseteq F_m$, we have $G_{j_q}\cup\{x_k\}\subseteq F_m$. Let $G_{j_q}\cup\{x_k\}=G_{\lambda}$ and $G_{j_{\ell}}\setminus\{x_k\}=G_{\gamma}$. Then  $y_{j_{\ell}}y_{j_q}-y_{\gamma}y_{\lambda}=g_{j_{\ell},j_q,x_k}$ with $\ini_<(g_{j_{\ell},j_q,x_k})=y_{j_{\ell}}y_{j_q}$. Since $y_{j_{\ell}}y_{j_q}$ divides $v$, this implies that $v\in I_{\Delta}$, a contradiction. So   
	$G_{i_q}\subseteq G_{j_q}$. Similarly, $G_{j_q}\subseteq G_{i_q}$. Hence we have $i_q=j_q$. By canceling $x_{G_{i_q}}$ from both sides of the equality~(\ref{eq1}) and repeating the same argument step by step we get $i_k=j_k$ for all $q\leq k\leq p$. Thus $w_1=w_2$ and then $\pi(u')=\pi(v')$ with $u',v'\notin I_{\Delta'}$. By induction hypothesis we obtain $u'=v'$. Hence $u=v$, as desired.   
	Therefore $J_{\Delta}$ has a quadratic Gr\"obner basis. It follows  that $A$ is Koszul, see \cite[Theorem 2.28]{HHO}. Moreover, since the initial ideal of  $J_{\Delta}$  is squarefree, by a theorem of Sturmfels \cite{St} (see also \cite[Corollary 4.26]{HHO}), $R_{\Delta}$ is  normal. Applying a result of Hochster \cite{H} (see also \cite[Theorem 6.3.5]{BH}) we conclude that $R_{\Delta}$ is Cohen--Macaulay. 
\end{proof}

In general, for an arbitrary simplicial complex $\Delta$ the ideal $(t)$ need not to be radical. For example, if $\Delta=\langle\{1,2\},\{2,3\}, \{3,1\}\rangle$, then it can be easily checked with Macaulay2 that $(t)$ is not a radical ideal in $R_\Delta$. As we see in Theorem~\ref{thm:classgroup} and Corollary~\ref{ainvariant}, the ideal $(t)$ being radical will be useful to obtain the class group of $R_{\Delta}$ and to compute the $a$-invariant of $R_\Delta$. Thanks to the Gr\"obner basis introduced in the proof of Theorem~\ref{powers} for quasi-forests we have

\begin{Proposition}\label{radical}
Let $\Delta$ be a quasi-forest, and let $t$  be the generator of $R_\Delta$ corresponding to the empty face of $\Delta$. Then $(t)$ is a radical ideal.
\end{Proposition}

\begin{proof}
Set $R=R_\Delta$. Since $(t)$  is a monomial ideal of $R$, it follows that the minimal prime ideals of $(t)$ are monomial prime ideals, that is, prime ideals generated by elements of the form $x_Ft$ with $F\in \Delta$. It follows that the radical ideal  of $(t)$ is generated  by monomials in $R$. Thus, in order to prove that $(t)$ is a radical ideal, it is enough to show that if  $t$ divides $u^k$ for a monomial $u$ and a positive integer $k$, then $t$ divides $u$.

Let $T$, $\pi: T\to R$, $<$, $f_{i,j}$ and  $g_{i,j,x}$ be as in the proof of Theorem~\ref{powers}. For any monomial $u\in R$, there exists a unique standard monomial $f_u=y_{i_1}\cdots y_{i_r}\in T$ with $\pi(f_u)=u$.  We claim that $t$ divides $u$ if and only if $y_1$ divides $f_u$.  If $y_1$ divides $f_u$, then clearly $t$ divides $u$, since $\pi(y_1)=t$. 

Conversely, suppose that $t$ divides $u$. Assume that $f_u=y_{i_1}\cdots y_{i_r}$ and by contradiction assume that $1\notin \{i_1,\ldots,i_r\}$. 
Since $t$ divides $u$, there exists a monomial $g=y_1y_{j_1}\cdots y_{j_{r-1}}\in T$ such that $\pi(g)=u$. We may assume that $h$ is the smallest monomial with respect to the order $<$ among such monomials $g$ (with $\pi(g)=u$ and $y_1$ divides $g$).  
Then $\pi(h)=u=\pi(f_u)$. Hence $h-f_u\in J_{\Delta}$ and since $f_u$ is a standard monomial, we have $\ini_<(h-f_u)=h$. Hence 
there exists a binomial $y_iy_j-y_ry_s$ in the Gr\"obner basis of $J_{\Delta}$ (with respect to $<$) which is of type $f_{i,j}$ or  $g_{i,j,x}$ with $y_iy_j=\ini_<(y_iy_j-y_ry_s)$ and such that $y_iy_j$ divides $h$. By the definition of  $f_{i,j}$ and  $g_{i,j,x}$ and the order $<$ we see that $i\neq 1$ and $j\neq 1$. Set $h'=(h/y_iy_j)y_ry_s$. Then $y_1$ divides $h'$ as well and $\pi(h')=u$. But $h'<h$ which contradicts to the choice of $h$, and hence $y_1$ divides $f_u$, as claimed.  

Next we show that for any positive integer $k$, $f_{u^k}=(f_u)^k$. Indeed, if $(f_u)^k\in  \ini_<(J_{\Delta})$, then by Theorem~\ref{powers}, $(f_u)^k$ is divided by a squarefree monomial $y_iy_j\in \ini_<(J_{\Delta})$. Therefore $y_iy_j$ divides $f_u$ as well, which contradicts to $f_u$ being standard. So $(f_u)^k\notin  \ini_<(J_{\Delta})$. This together with $\pi((f_u)^k)=u^k$ implies that $f_{u^k}=(f_u)^k$.                 

Now, suppose $u^k\in (t)$ for some positive integer $k$. Then $t$ divides $u^k$. As was shown above this implies that $y_1$ divides $f_{u^k}$. The equality $f_{u^k}=(f_u)^k$ implies that $y_1$ divides $f_u$ too. Hence $t$ divides $u$, which means that $u\in (t)$.
Hence $(t)$ is a radical ideal. 
\end{proof}

\begin{Corollary}\label{verygoodquasi1}
Let $\Delta$ be a quasi-forest on $[n]$ and let $C_1, \ldots, C_r$ be the minimal vertex covers of $G_\Delta$. Then 
\begin{enumerate}
	\item[{\em(i)}] $[\omega_{R_{\Delta}}]= \sum_{j=1}^r (n-|C_j|) [P_{C_j}].$
	 
	\item[{\em(ii)}] the $a$-invariant of $R_\Delta$ is $-(c_\Delta+1)$.   	 
\end{enumerate}
	
\end{Corollary}

\begin{proof}
By Proposition~\ref{radical}, $(t)$ is a radical ideal. Hence by Theorem~\ref{healthybread}  we have $(t)=\bigcap_{i=1}^r P_{C_i}$ and then $\sum_{j=1}^r  [P_j]=0$. This together with Theorem~\ref{verygood} implies (i). Part (ii) follows from Theorem~\ref{ainvariant}. 
\end{proof}



By using the description of the class of $\omega_{R_{\Delta}}$ in Corollary~\ref{verygoodquasi1} we can characterize 
Gorenstein toric rings of quasi-forests. This characterization also follows from \cite{HO2}. 

\begin{Corollary}\label{gor}
Let $\Delta$ be a quasi-forest on $[n]$. 
Then   
the following conditions are equivalent.
\begin{enumerate}
	\item[{\em(i)}] $R_\Delta$ is Gorenstein. 
		\item[{\em(ii)}]  $G_\Delta$ is unmixed. 
	\item[{\em(iii)}] $G_\Delta$ is Cohen-Macaulay.
	\item[{\em(iv)}]	$[n]$ is the disjoint union of facets of $\Delta$ which admit a free vertex.
\end{enumerate}
\end{Corollary}

\begin{proof}
Let $C_1, \ldots, C_r$ be the minimal vertex covers of $G_\Delta$ and $P_i=P_{C_i}$ for $i=1, \ldots, r$.
Before proving the equivalence of the above statements, we first characterize the divisorial ideal in $R_{\Delta}$ whose class in $\Cl (R_{\Delta})$ is zero.  Let $I$ be any divisorial ideal in $R_{\Delta}$. By Corollary~\ref{perfectclass}, $\Cl (R_{\Delta})$ is generated by the classes of $[P_j]$, $j=1, \ldots, r$. Hence we have $[I]= \sum_{j=1}^r c_j [P_j]$ for some $c_j \in\ZZ$. Furthermore, $[P_r]=-\sum_{i=1}^{r-1}[P_i]$.  Therefore $[I]=  \sum_{j=1}^{r-1} (c_j-c_r) [P_j]$.  By Corollary~\ref{perfectclass}, $[P_1], \ldots, [P_{r-1}]$ form a basis of  $\Cl (R_{\Delta})$. Therefore, $[I]=0$ if and only if   $c_j=c_r$ for all $j$. 

(i) \iff (ii):   Note that $R_\Delta$ is Gorenstein if and only if $\omega_{R_{\Delta}}\iso R_{\Delta}$, and this is the case if and only if $[\omega_{R_{\Delta}}]=0$. Then  Corollary~\ref{verygoodquasi1}(i) together with the above argument shows that  $[\omega_{R_{\Delta}}]=0$ if and only if all minimal vertex covers of $G_\Delta$ are of the same cardinality. 

By \cite[Theorem 3.3]{HHZ1}, $G_{\Delta}$ is a chordal graph. The equivalence of (ii), (iii) and (iv) for chordal graphs is shown in \cite{HHZ}.
\end{proof}

	\end{document}